\documentclass[12pt,a4paper]{article}
\usepackage{latexsym,amssymb,amsfonts,amsmath,amsthm,nccmath,float,enumitem,setspace,bm,authblk,graphicx,arydshln}
\usepackage[usenames,dvipsnames]{xcolor}
\usepackage[hidelinks]{hyperref}
\usepackage[margin=2cm]{geometry}
\usepackage{booktabs}

\hypersetup{
    colorlinks,
    linkcolor={blue!80!black},
    citecolor={blue!80!black},
    urlcolor={blue!80!black}
}

\newtheorem{Theorem}{Theorem}[section]
\newtheorem{Proposition}{Proposition}[section]

\newtheorem{Lemma}{Lemma}[section]

\def \leq {\leqslant}
\def \geq {\geqslant}

\let\oldproofname=\proofname
\renewcommand{\proofname}{\rm\bf{\oldproofname}}

\numberwithin{equation}{section}

\allowdisplaybreaks

\begin{document}

\title{Fractional clique decompositions of dense balanced multipartite graphs}

\author[a,b]{Tao Feng \thanks{Supported by NSFC under Grant 12271023}}
\author[a]{Hengrui Liu}
\author[a]{Shikang Yu}
\affil[a]{School of Mathematics and Statistics, Beijing Jiaotong University, Beijing, 100044, P.R. China}
\affil[b]{Hebei Provincial Key Laboratory of Mathematical Theory and Analysis for Network and Data Science, Beijing Jiaotong University, Beijing, 100044, P.R. China}
\renewcommand*{\Affilfont}{\small\it}
\renewcommand\Authands{ and }

\affil[ ]{tfeng@bjtu.edu.cn, henryleo@bjtu.edu.cn, healthyu@bjtu.edu.cn}
\date{}

\maketitle
\underline{}
\begin{abstract}
This paper concerns fractional $K_s$-decompositions of multipartite graphs. For integers $r\ge s\ge 3$, we consider balanced $r$-partite graphs $G$ on $rn$ vertices. We establish necessary conditions for $G$ to admit a fractional $K_s$-decomposition, extending the notion of $s$-admissibility from the case $r=s$ to $r>s$. Using an association scheme on the edge set of a complete $r$-partite graph, we prove that if $r\ge s+2$ and the partite minimum degree of $G$ is at least $(1-c)n$ with $c\le 1/((s-2)(s+1)(s-1)^4)$, then $G$ has a fractional $K_s$-decomposition. For $r=s+1$, we show that under the condition $c\le 1/(3s^3(s-2)^2)$, every $s$-admissible balanced $(s+1)$-partite graph with partite minimum degree at least $(1-c)n$ admits a fractional $K_s$-decomposition. These results provide new degree thresholds for fractional $K_s$-decompositions of multipartite graphs with more than $s$ parts.
\end{abstract}


\section{Introduction}\label{sec:introduction}

All graphs considered here are finite, undirected and simple unless otherwise specified. Let $H$ be a graph. Denote by $V(H)$ the vertex set of $H$ and by $E(H)$ the edge set of $H$. For $v \in V(H)$, let $d(v)$ denote the degree of $v$ and $N(v)$ denote the set of all neighbours of $v$. For $v \in V(H)$ and $U \subseteq V(H)$, let $d(v, U)$ denote the number of neighbors of $v$ in $U$. For $U_1,U_2 \subset V(H)$ and $U_1\cap U_2=\emptyset$, denote by $E(U_1, U_2)$ the set of edges in $H$ with one vertex in $U_1$ and the other vertex in $U_2$. Let $\gcd(H)$ denote the largest integer that divides the degree of each vertex of $H$. For a positive integer $k$, let $K_k$ denote the complete graph on $k$ vertices.

Let $F$ be a subgraph of a graph $G$. The subgraphs of $G$ isomorphic to $F$ are referred to as \emph{copies} of $F$. A classical question in graph theory is whether the edges of $G$ can be partitioned into edge disjoint copies of $F$; such a partition of $E(G)$ is called an \emph{$F$-decomposition} of $G$. $G$ is said to be \emph{$F$-divisible} if $|E(F)|$ divides $|E(G)|$ and $\gcd(F)$ divides $\gcd(G)$. Kirkman \cite{Kirkman} showed that every $K_3$-divisible complete graph admits a $K_3$-decomposition. Wilson \cite{Wilson75} generalized this by showing that for any graph $F$, every sufficiently large $F$-divisible complete graph admits an $F$-decomposition. Keevash \cite{Keevash14} generalized this to hypergraphs using randomized algebraic constructions, settling a long-standing question regarding the existence of designs (see \cite{GKLO} for a different proof via the absorbing method).

A further research problem is to find $F$-decompositions in $F$-divisible graphs with high minimum degree. This problem is closely related to finding a fractional $F$-decomposition in graphs with high minimum degree (cf. \cite{BKLO,GKLMO}). A graph $G$ admits a {\em fractional $F$-decomposition} if a nonnegative weighting can be given to the copies of $F$ in $G$ so that each edge lies in copies of $F$ with total weight 1. That is, if $\mathcal{F}(G)$ is the set of copies of $F$ in $G$, then there exists a nonnegative weight function $w : \mathcal{F}(G) \to [0,1]$ such that, for each edge $e \in E(G)$, $\sum_{F \in \mathcal{F}(G):\, e \in E(F)} w(F) = 1$. For example, for any $m\geq s\geq 2$, every complete graph $K_m$ has a fractional $K_s$-decomposition by weighting all the copies of $K_s$ in $K_m$ by $1/\binom{m-2}{s-2}$. Clearly, an $F$-decomposition is a fractional $F$-decomposition with weights either $0$ or $1$.

A central problem in the study of fractional decompositions is determining the minimal value $c_F$ such that every $m$-vertex graph $G$ with minimum degree at least $c_F m$ admits a fractional $F$-decomposition. Here we restrict our attention to the case where $F$ is the complete graph $K_s$. Regarding lower bounds on $c_{K_s}$, Yuster \cite[Section 3]{Yuster} showed that $c_{K_s} \geq 1-\frac{1}{s+1}$ for any $s \geq 3$, and very recently, Delcourt, Henderson, Lesgourgues, and Postle \cite{Beyond} demonstrated that for each $s \geq 4$, there exists some $\alpha > 1$ such that $c_{K_s} \geq 1 - \frac{1}{\alpha(s+1)}$. On the other hand, there have been a series of improvements to the upper bounds on $c_{K_s}$ for $s \geq 3$. Yuster \cite{Yuster} established that for all sufficiently large $m$-vertex graphs $G$, $c_{K_s} \leq 1-\frac{1}{9s^{10}}+o(1)$; Dukes \cite{Dukes12} improved this to $c_{K_s} \leq 1-\frac{2}{9s^2(s-1)^2}$; and Barber, K\"{u}hn, Lo, Montgomery, and Osthus \cite{BKLMO} further showed that $c_{K_s} \leq 1-\frac{1}{10^4s^{3/2}}$ for $m\geq 10^4 s^3$. Montgomery \cite{Montgomery19} obtained the currently best-known bound $c_{K_s} \leq 1-\frac{1}{100s}$ (without requiring $m$ to be sufficiently large), which is the first bound with a denominator linear in $s$. For the specific case of $s=3$, for sufficiently large $m$-vertex graphs $G$, the bounds have been sequentially improved: Garaschuk \cite{Garaschuk} proved that $c_{K_3} \leq \frac{22}{23}$; Dross \cite{Dross} showed that $c_{K_3} \leq 0.9+o(1)$; and Dukes and Horsley \cite{DH} demonstrated that $c_{K_3} \leq 0.852$. Delcourt and Postle \cite{DP21} established the currently best-known bound $c_{K_3} \leq 0.82733$, which holds without requiring $m$ to be sufficiently large.

This paper focuses on the problem of fractional clique decompositions of multipartite graphs. For an integer $r \geq 2$, an $r$-partite graph $H$ with vertex partition into parts $V_1, \ldots, V_r$ is called {\em balanced} if $|V_1| = \cdots = |V_r|$. The \emph{partite minimum degree} of $H$ is defined as
$$
\hat\delta(H) := \min \{ d(v, V_i) : 1 \leq i \leq r,\; v \in V(H) \setminus V_i \}.
$$

Let $G$ be a balanced $s$-partite graph with $s\geq 3$. If $G$ admits a fractional $K_s$-decomposition, then
\begin{align}\label{eq:locally}
d(v, V_i) = d(v, V_j)
\end{align}
for every $1 \leq i<j \leq s$ and every vertex $v \notin V_i \cup V_j$ (see \cite[Section 1]{Montgomery17}). We call such a graph $G$ \emph{$s$-admissible}, while \cite{Montgomery17} and \cite{BD} use \emph{$K_s$-divisible} and \emph{locally balanced}, respectively.
A fundamental problem is determining the minimal value $\hat c_{K_s}$ such that every $s$-admissible balanced $s$-partite graph on $sn$ vertices with partite minimum degree at least $\hat c_{K_s} n$ admits a fractional $K_s$-decomposition. Barber, K\"{u}hn, Lo, Osthus, and Taylor \cite{BKLOT} pointed out that $\hat c_{K_s} \geq 1-\frac{1}{s+1}$ for all $s \geq 3$, and recently, Delcourt, Henderson, Lesgourgues, and Postle \cite{Beyond} demonstrated that for each $s \geq 4$, there exists some $\beta > 1$ such that $\hat c_{K_s} \geq 1 - \frac{1}{\beta(s+1)}$. Bowditch and Dukes \cite{BD} proved that $\hat c_{K_3} \leq \frac{24}{25}$ for all sufficiently large $3$-admissible balanced $3$-partite graphs $G$. Montgomery \cite{Montgomery17} extended this result to the general case of $K_s$ for $s \geq 3$, which holds
without requiring $G$ to be sufficiently large.

\begin{Theorem}\label{thm:DM}\emph{\cite{BD, Montgomery17}}
Let $G$ be an $s$-admissible balanced $s$-partite graph with $sn$ vertices.
\begin{enumerate}
    \item[\textup{(i)}] If $s=3$, $n$ is sufficiently large, and $\hat\delta(G)\geq \frac{24}{25}n$, then $G$ admits a fractional $K_3$-decomposition.
    \item[\textup{(ii)}] If $s\geq 3$, $n$ is a positive integer, and $\hat\delta(G) \geq (1 - \frac{1}{10^6 s^3})n$, then $G$ admits a fractional $K_s$-decomposition.
\end{enumerate}
\end{Theorem}

A more general problem, as noted in both \cite{BD} and \cite{Montgomery17}, concerns the partite minimum degree required in a balanced $r$-partite graph to ensure the existence of a fractional $K_s$-decomposition for $3 \leq s \leq r$. To address this problem, we establish the following necessary conditions for an $r$-partite graph (not necessarily balanced) to admit a fractional $K_s$-decomposition in Section \ref{subsec:nece}.

\begin{Proposition}\label{prop:nece}
Let $r\geq s\geq 3$. Let $G$ be an $r$-partite graph with vertex partition into parts $V_1, \dots, V_r$ that admits a fractional $K_s$-decomposition. Then
\begin{align}\label{eq:nec-1}
d(v,V_k) \leq \frac{d(v)}{s-1}
\end{align}
for all $v \in V(G)$ and $1 \leq k \leq r$. Furthermore, if $r = s+1$, then
\begin{align}\label{eq:nec-2}
|E(V_i,V_j)| = \frac{d_i + d_j}{s-1} - \frac{|E(G)|}{\binom{s}{2}}
\end{align}
for all $1 \leq i < j \leq s+1$, where $d_\ell = \sum_{k=1, k \neq \ell}^{s+1} |E(V_\ell, V_k)|$ with $1\leq \ell\leq s+1$.
\end{Proposition}

When $r=s$, \eqref{eq:nec-1} implies \eqref{eq:locally}, and thus yields that $G$ is $s$-admissible. We also call an $r$-partite graph $G$ (with $r>s$) \emph{$s$-admissible} if it satisfies \eqref{eq:nec-1}, and additionally \eqref{eq:nec-2} when $r=s+1$.

Following the technique developed by Bowditch and Dukes \cite{BD}, we investigate the partite minimum degree for fractional $K_s$-decompositions in balanced $r$-partite graphs, where $r>s\geq 3$, by exploring the association scheme defined on the edge set of a balanced complete $r$-partite graph. We establish the following result in Section \ref{sec:s+2}.

\begin{Theorem}\label{thm:main1}
Let $s \geq 3$ and $r \geq s + 2$. Let $G$ be a balanced $r$-partite graph on $rn$ vertices with $\hat\delta(G) \geq (1-c)n$, where
\begin{align}\label{eq:c-imply}
c \leq \frac{1}{(s-2)(s+1)(s-1)^4}.
\end{align}
Then $G$ admits a fractional $K_s$-decomposition.
\end{Theorem}

In Theorem \ref{thm:main1}, although $G$ is not explicitly stated to be $s$-admissible, the condition \eqref{eq:c-imply} implies its $s$-admissibility. Specifically, for any vertex $v\in V(G)$, the partite minimum degree condition $\hat\delta(G) \geq (1-c)n$ ensures that $d(v)\geq (1-c)n(r-1)$. Thus
$$\frac{d(v)}{s-1}\geq \frac{(1-c)n(r-1)}{s-1}\geq n\geq d(v,V_k)$$
for all $1 \leq k \leq r$, which verifies \eqref{eq:nec-1}.

A more precise version of the upper bound of $c$ in Theorem \ref{thm:main1} is \eqref{eq:c-1} in Lemma \ref{prop:solution1}, but this upper bound depends on both $r$ and $s$.

When $r=s+1$, this case is more complicated. Montgomery \cite[Section 1]{Montgomery17} noted that there may not be a simple set of divisibility conditions distinguishing which large graphs with a high minimum degree have a fractional $K_s$-decomposition. Nevertheless, we provide necessary conditions in Proposition \ref{prop:nece}. Based on these conditions, we establish the following result in Section \ref{sec:s+1}.

\begin{Theorem}\label{thm:main2}
Let $s \geq 3$. Let $G$ be an $s$-admissible balanced $(s+1)$-partite graph on $(s+1)n$ vertices with $\hat\delta(G) \geq (1-c)n$, where
\begin{align}\label{eq:c33}
c \leq \frac{1}{3s^3(s-2)^2}.
\end{align}
Then $G$ admits a fractional $K_s$-decomposition.
\end{Theorem}

Using the same arguments as those after Theorem \ref{thm:main1}, one can see that for the graph $G$ in Theorem \ref{thm:main2}, the condition \eqref{eq:nec-1} is automatically satisfied because of \eqref{eq:c33}. The reader will see later that in the proof of Theorem \ref{thm:main2} we will use the condition \eqref{eq:nec-2} (see Lemma \ref{prop:fromdukes1}, which uses Lemma \ref{cor:nece}). Therefore, in the statement of Theorem \ref{thm:main2} we explicitly require that $G$ is $s$-admissible.

A more precise version of the upper bound of $c$ in Theorem \ref{thm:main2} is \eqref{eq:c-22} in Lemma \ref{prop:solution2}.

In fact, there is a simple way to quickly derive a fractional $K_s$-decomposition from Theorem \ref{thm:DM} for an $r$-admissible balanced $r$-partite graph for any $r>s$. Specifically, if an $r$-admissible balanced $r$-partite graph $G$ (which satisfies \eqref{eq:locally}) admits a fractional $K_r$-decomposition, then since every $K_r$ has a fractional $K_s$-decomposition, we can decompose each $K_r$ in the fractional $K_r$-decomposition of $G$ into copies of $K_s$. This yields a fractional $K_s$-decomposition of $G$. Hence as a straightforward corollary of Theorem \ref{thm:DM}(ii), we have that for an $r$-admissible balanced $r$-partite graph $G$ with $rn$ vertices, if $3\leq s\leq r$ and $\hat\delta(G) \geq (1 - \frac{1}{10^6 r^3})n$, then $G$ admits a fractional $K_s$-decomposition. However, by Proposition \ref{prop:nece}, the necessary condition required for a balanced $r$-partite graph $G$ to admit a fractional $K_s$-decomposition is that $G$ is $s$-admissible, not $r$-admissible; the former condition is weaker. Therefore, the requirement for the graph $G$ in our Theorems \ref{thm:main1} and \ref{thm:main2} is much weaker. On the other hand, regarding Theorem \ref{thm:main1}, the upper bound for $c$ is sharper than $\frac{1}{10^6 r^3}$ for all $s+2 \le r \le 105$. Moreover, for $r > 105$, the upper bound in Theorem \ref{thm:main1} remains strictly sharper than $\frac{1}{10^6 r^3}$ for all $s$ satisfying $3 \le s \le 10 \sqrt{r}$. As for Theorem \ref{thm:main2}, in the specific case where $r = s+1$, the upper bound provides a strict improvement over $\frac{1}{10^6 r^3}$ for all $r \le 581$.

Theorems \ref{thm:main1} and \ref{thm:main2} can be used to establish the asymptotic existence of an $\varepsilon$-approximate $K_s$-decomposition of an $s$-admissible balanced $r$-partite graph. Let $\varepsilon$ be a real number satisfying $0 < \varepsilon \le 1$. An \emph{$\varepsilon$-approximate $F$-decomposition} of a graph $G$ is a collection of edge-disjoint copies of $F$ covering all but at most $\varepsilon|E(G)|$ edges of $G$. Haxell and R\"{o}dl \cite{HR} demonstrated that for any fixed graph $F$ and any $\varepsilon>0$, there exists an integer $m_0$ such that every graph $G$ on $m\ge m_0$ vertices admitting a fractional $F$-decomposition also admits an $\varepsilon$-approximate $F$-decomposition. Thus any balanced $r$-partite graph $G$ satisfying the conditions of Theorems \ref{thm:main1} and \ref{thm:main2} admits an $\varepsilon$-approximate $K_s$-decomposition.

The remainder of this paper is organized as follows. Section \ref{subsec:nece} establishes necessary conditions for fractional $K_s$-decompositions of $r$-partite graphs. Section \ref{subsec:reduction} reduces the problem of finding a fractional $K_s$-decomposition of a balanced $r$-partite graph to solving a system of linear equations. Section \ref{sec:scheme} characterizes the association scheme defined on the edge set of a balanced complete $r$-partite graph. The proofs of Theorem~\ref{thm:main1} and Theorem~\ref{thm:main2} are provided in Sections \ref{sec:s+2} and \ref{sec:s+1}, respectively. Concluding remarks are given in Section \ref{sec:conc}.

\section{Necessary conditions for fractional $K_s$-decompositions of $r$-partite graphs} \label{subsec:nece}

This section gives the proof of Proposition \ref{prop:nece}, thereby providing necessary conditions for fractional $K_s$-decompositions of $r$-partite graphs (not necessarily balanced) with $r\geq s\geq 3$.

\begin{proof}[\bf{Proof of Proposition \ref{prop:nece}}]

Let $\mathcal{K}$ denote the set of all copies of $K_s$ in $G$. By the definition of a fractional $K_s$-decomposition, there exists a weight function $\omega: \mathcal{K} \to [0, 1]$ such that for every edge $e \in E(G)$, $\sum_{K \in \mathcal{K}:\, e \in E(K)} \omega(K) = 1$. Fix an arbitrary vertex $v \in V_i$ for some $i \in \{ 1,\ldots,r \}$, and define $\mathcal{K}_v = \{K \in \mathcal{K} : v \in V(K)\}$. Then
\begin{align*}
d(v) = \sum_{u \in N(v)} 1 &= \sum_{u \in N(v)} \left( \sum_{K \in \mathcal{K}: \{v,u\} \in E(K)} \omega(K) \right) \\
&= \sum_{K \in \mathcal{K}_v}  \sum_{u \in V(K) \setminus \{v\}}  \omega(K)  = (s-1) \sum_{K \in \mathcal{K}_v} \omega(K).
\end{align*}
Therefore, for any $k \in \{ 1,\ldots,r \} \setminus \{ i \}$, we have
\begin{align*}
d(v, V_k) &= \sum_{u \in V_k \cap N(v)} 1 = \sum_{u \in V_k \cap N(v)} \left( \sum_{K \in \mathcal{K}: \{v,u\} \in E(K)} \omega(K) \right) \\
&= \sum_{K \in \mathcal{K}_v} \omega(K) | V(K)\cap V_k | \leq \sum_{K \in \mathcal{K}_v} \omega(K) = \frac{d(v)}{s-1}.
\end{align*}
For $k = i$, it is trivial that $d(v, V_i) = 0 \leq \frac{d(v)}{s-1}$. Therefore, \eqref{eq:nec-1} holds for any $1 \leq k \leq r$.

Now, suppose $r=s+1$. Let $t = \sum_{K \in \mathcal{K}} \omega(K)$. Then
\begin{align*}
|E(G)| = \sum_{e \in E(G)} 1 = \sum_{e \in E(G)} \left( \sum_{K \in \mathcal{K}: e \in E(K)} \omega(K) \right) = \sum_{K \in \mathcal{K}} \sum_{e \in E(K)} \omega(K)= \binom{s}{2} t,
\end{align*}
so $t = |E(G)|/\binom{s}{2}$. Since $G$ is an $(s+1)$-partite graph, every copy of $K_s$ in $G$ intersects exactly $s$ distinct parts of $G$, and thus is disjoint from exactly one part. For $i \in \{ 1,\ldots, s+1 \} $, let $\mathcal{S}_i = \{K \in \mathcal{K}: V(K)\cap V_i=\emptyset\}$, and set $t_i=\sum_{K \in \mathcal{S}_i } \omega(K)$. Since $\mathcal{K}$ is the disjoint union of $\mathcal{S}_1, \ldots, \mathcal{S}_{s+1}$, we have $\sum_{i=1}^{s+1} t_i = t$.

For any $1 \leq i < j \leq s+1$, the number of edges between $V_i$ and $V_j$ can be calculated as
\begin{align*}
| E(V_i, V_j) | &= \sum_{e \in E(V_i, V_j)} \left( \sum_{K \in \mathcal{K}: e \in E(K)} \omega(K) \right) = \sum_{K \in \mathcal{K}} |E(K) \cap E(V_i, V_j)| \omega(K) \\
&= \sum_{\ell=1}^{s+1} \sum_{K \in \mathcal{S}_\ell} |E(K) \cap E(V_i, V_j)| \omega(K) = \sum_{\ell \notin \{i, j\}} \sum_{K \in \mathcal{S}_\ell} \omega(K) = \sum_{\ell \notin \{i, j\}} t_\ell  =  t - t_i - t_j.
\end{align*}
It follows that for any $1\leq \ell \leq s+1$,
\begin{align*}
d_\ell = \sum_{\substack{k=1\\ k\neq \ell}}^{s+1} |E(V_\ell, V_k)| =
\sum_{\substack{k=1\\ k\neq \ell}}^{s+1} (t - t_\ell - t_k) =
 s (t - t_\ell) - \sum_{\substack{k=1\\ k\neq \ell}}^{s+1} t_k = (s-1)(t - t_\ell),
\end{align*}
which yields $t_\ell = t - \frac{d_\ell}{s-1}$. Substituting the expressions for $t_i$ and $t_j$ into the equation for $|E(V_i, V_j)|$, we obtain
$$ | E(V_i, V_j)| = t - \left( t - \frac{d_i}{s-1} \right) - \left( t - \frac{d_j}{s-1} \right) = \frac{d_i + d_j}{s-1} - t = \frac{d_i + d_j}{s-1} - \frac{|E(G)|}{\binom{s}{2}}. $$
Therefore, \eqref{eq:nec-2} holds for all $1 \leq i < j \leq s+1$.
\end{proof}

The following lemma will be used to prove Theorem \ref{thm:main2} in Section \ref{sec:s+1}.

\begin{Lemma}\label{cor:nece}
Let $s\geq 3$ and $G$ be an $(s+1)$-partite graph with vertex partition $\mathcal{P} = \{V_1,\dots, V_{s+1}\}$. Let $N$ be an $\binom{s+1}{2} \times \binom{s+1}{s}$ matrix whose rows and columns are indexed by the $2$-subsets $T$ and $s$-subsets $K$ of $\mathcal{P}$, respectively, where the $(T,K)$-entry is $1$ if $T \subseteq K$ and $0$ otherwise. Let $\mathbf{b}$ be the column vector indexed by the $2$-subsets of $\mathcal{P}$, with entries $\mathbf{b}(\{V_i,V_j\})=|E(V_i,V_j)|$ for $1\leq i<j\leq s+1$. If $G$ is $s$-admissible, then the linear system $N\mathbf{x} = \mathbf{b}$ admits a solution $\mathbf{x} = (x_1, \dots, x_{s+1})^\top$, where $x_\ell = \frac{|E(G)|}{\binom{s}{2}} - \frac{d_\ell}{s-1}\geq 0$ and $d_\ell = \sum_{k=1, k \neq \ell}^{s+1} |E(V_\ell, V_k)|$ for each $1\leq \ell \leq s+1$.
\end{Lemma}

\begin{proof}
Without loss of generality, assume that for each $1 \leq \ell \leq s+1$, the $\ell$-th column of $N$ is indexed by the $s$-subset $\mathcal{P} \setminus \{ V_\ell \}$. Therefore, the equation in the linear system $N\mathbf{x} = \mathbf{b}$ corresponding to the row $\{V_i, V_j\}$ can be written as:
$$\sum_{\substack{\ell=1\\ \ell \notin \{i, j\}}}^{s+1} x_\ell = |E(V_i, V_j)|.$$
We claim that $x_\ell = \frac{|E(G)|}{\binom{s}{2}} - \frac{d_\ell}{s-1}$ is a solution to the above equation. Indeed,
$$\sum_{\substack{\ell=1\\ \ell \notin \{i, j\}}}^{s+1} x_\ell =
\sum_{\substack{\ell=1\\ \ell \notin \{i, j\}}}^{s+1} \left( \frac{|E(G)|}{\binom{s}{2}} - \frac{d_\ell}{s-1} \right) = (s-1)\frac{|E(G)|}{\binom{s}{2}} - \frac{1}{s-1} \sum_{\substack{\ell=1\\ \ell \notin \{i, j\}}}^{s+1} d_\ell.$$
Since $\sum_{\ell=1}^{s+1} d_\ell = 2|E(G)|$, we have
$$\sum_{\substack{\ell=1\\ \ell \notin \{i, j\}}}^{s+1} x_\ell = (s-1)\frac{|E(G)|}{\binom{s}{2}} - \frac{2|E(G)| - d_i - d_j}{s-1} = \frac{d_i + d_j}{s-1} - \frac{|E(G)|}{\binom{s}{2}} = |E(V_i,V_j)|,$$
where the last equality holds since $G$ is $s$-admissible (using \eqref{eq:nec-2}).

Now it remains to verify that $x_\ell\geq 0$ for each $1\leq \ell\leq s+1$. For $1\leq k\leq s+1$ and $k\neq \ell$, we have
\begin{align}\label{eq:e11}
|E(V_k,V_\ell)|= \sum_{v \in V_k} d(v, V_\ell) \leq \sum_{v \in V_k} \frac{d(v)}{s-1} = \frac{d_k}{s-1},
\end{align}
where the inequality holds since $G$ is $s$-admissible (using \eqref{eq:nec-1}). By \eqref{eq:nec-2},
\begin{align}\label{eq:e12}
|E(V_k,V_\ell)|=\frac{d_k + d_\ell}{s-1} - \frac{|E(G)|}{\binom{s}{2}}.
\end{align}
Combining \eqref{eq:e11} and \eqref{eq:e12}, we have
$$x_\ell = \frac{|E(G)|}{\binom{s}{2}} - \frac{d_\ell}{s-1} \geq 0.$$
This completes the proof.
\end{proof}

\section{Problem reduction}\label{subsec:reduction}

This section employs the same technique as in \cite{BD} to reduce the problem of finding a fractional $K_s$-decomposition of a balanced $r$-partite graph to solving a system of linear equations.

Let $H$ be a graph and $s\geq3$ be an integer. Let $\mathcal{K}_s(H)$ denote the set of all copies of $K_s$ in $H$. Let $W_H$ be the matrix whose rows are indexed by edges $e \in E(H)$ and columns by $K_s \in \mathcal{K}_s(H)$, where the $(e,K_s)$-entry is $1$ if $e \in E(K_s)$ and $0$ otherwise.
For any real vector $\mathbf{x}$, the notation $\mathbf{x} \geq 0$ means that all entries of $\mathbf{x}$ are non-negative. Let $\mathbf{1}$ denote the all-ones vector, with dimension implicit from the context. The existence of a fractional $K_s$-decomposition of $H$ is equivalent to the existence of a vector $\mathbf{x} \geq 0$ satisfying $W_H \mathbf{x} = \mathbf{1}$.

Let
\begin{align}\label{eq:M_H}
M_H := W_H W_H^\top.
\end{align}
Then $M_H$ is a matrix with rows and columns indexed by $E(H)$, and its $(e,e')$-entry equals the number of copies of $K_s$ in $H$ containing both edges $e$ and $e'$. If the linear system $M_H \mathbf{y} = \mathbf{1}$ admits a solution $\mathbf{y}\geq0$, then the vector $\mathbf{x} = W_H^\top \mathbf{y}\geq0$ satisfies $W_H \mathbf{x} = \mathbf{1}$, thereby yielding a fractional $K_s$-decomposition of $H$.

\begin{Lemma}\label{lem:frac}
Let $s\geq3$. If the linear system $M_H \mathbf{y} = \mathbf{1}$ admits a solution $\mathbf{y}\geq0$, then $H$ admits a fractional $K_s$-decomposition.
\end{Lemma}

Let $r\geq3$ be an integer. Throughout the rest of this paper, we assume that
\begin{itemize}
\item $\Gamma$ is a balanced complete $r$-partite graph on $rn$ vertices, and
\item $G$ is a spanning $r$-partite subgraph of $\Gamma$ with the same vertex partition.
\end{itemize}
Let $\overline{E(G)} = E(\Gamma) \setminus E(G)$. We adopt the convention that,
\begin{itemize}\item for any $|E(\Gamma)|\times |E(\Gamma)|$ matrix $N$ whose rows and columns are indexed by $E(\Gamma)$, we order the edges in $E(\Gamma)$ such that those in $E(G)$ appear first.
\end{itemize}For any subsets $X, Y \subseteq E(\Gamma)$, let $N[X, Y]$ denote the submatrix of $N$ induced by the rows indexed by $X$ and the columns indexed by $Y$.
Let $\Delta M_{(\Gamma,G)}$ be the matrix with rows and columns indexed by $E(\Gamma)$ given by
\begin{align}\label{eq:Delta_M}
\Delta M_{(\Gamma,G)} :=
\left(
\begin{array}{c:c}
M_G - M_\Gamma[E(G), E(G)] & -M_\Gamma[E(G), \overline{E(G)}] \\
\hdashline
\\[-2ex]
\multicolumn{2}{c}{O} \\
\end{array}\right).
\end{align}
Then
$$
M_\Gamma + \Delta M_{(\Gamma,G)} =
\left(
\begin{array}{c:c}
M_G & O \\
\hdashline
\\[-2ex]
\multicolumn{2}{c}{M_\Gamma[\overline{E(G)}, E(\Gamma)]} \\
\end{array}
\right)
$$
Thus a solution $\mathbf{z} \geq 0$ to the linear system $( M_\Gamma + \Delta M_{(\Gamma,G)} )\mathbf{z} = \mathbf{1}$ induces a solution $\mathbf{y} \geq 0$ to the linear system $M_G \mathbf{y} = \mathbf{1}$. By Lemma~\ref{lem:frac}, this leads to a fractional $K_s$-decomposition of $G$. 

\begin{Lemma}\label{lem:frac1}
Let $r\geq s\geq3$. If the linear system $(M_\Gamma + \Delta M_{(\Gamma,G)})\mathbf{z} = \mathbf{1}$ admits a solution $\mathbf{z} \geq 0$, then $G$ admits a fractional $K_s$-decomposition.
\end{Lemma}

Let $m$ be a positive integer and $\mathbb{R}^{m \times m}$ denote the set of all $m \times m$ matrices over $\mathbb{R}$. For $A \in \mathbb{R}^{m \times m}$, the \emph{infinity norm} $\|A\|_\infty$ of $A$ is defined as the maximum absolute row sum of $A$, i.e., $\|A\|_\infty = \max\limits_{1\leq i\leq m} \sum_{j=1}^m |A(i,j)|$. The following lemma can be used to determine whether a linear system admits a nonnegative solution.

\begin{Lemma}\label{prop:fromdukes}\emph{\cite[Corollary 3.3]{BD}}
Let $m$ be a positive integer, and $A, B \in \mathbb{R}^{m \times m}$. Assume that
\begin{enumerate}
\item[$(1)$] $A$ is invertible;
\item[$(2)$] the linear system $A\mathbf{x} = \mathbf{d}$ admits a unique solution $\mathbf{x}=a\mathbf{1}$, where $a\geq 0$ and $\mathbf{1}$ is the vector of all ones;
\item[$(3)$] $\|A^{-1}B\|_\infty \leq \frac{1}{2}$.
\end{enumerate}
Then $A + B$ is invertible and the linear system $(A + B)\mathbf{z} = \mathbf{d}$ admits a unique solution $\mathbf{z}\geq0$.
\end{Lemma}

By Lemmas \ref{lem:frac1} and \ref{prop:fromdukes}, to prove that $G$ admits a fractional $K_s$-decomposition, it suffices to show that the matrices $M_\Gamma$ and $\Delta M_{(\Gamma,G)}$ satisfy the conditions of Lemma \ref{prop:fromdukes}, where we set $A=M_\Gamma$ and $B=\Delta M_{(\Gamma,G)}$.
In Section \ref{sec:scheme}, we present an association scheme defined on $E(\Gamma)$. This scheme enables us to compute the spectrum of $M_\Gamma$, thereby establishing the invertibility of $M_\Gamma$ for any $r \geq s+2$, and helps us estimate the infinity norm $\| M_\Gamma^{-1} \Delta M_{(\Gamma,G)}\|_\infty$ in Section \ref{sec:s+2}.

\section{The spectrum of $M_\Gamma$}\label{sec:scheme}

In this section, we determine the spectrum of $M_\Gamma$ by showing that it belongs to a low-dimensional matrix algebra with an explicit description. We start by recalling necessary background on association schemes.

\subsection{Preliminaries on association schemes}

Association schemes form a central part of algebraic combinatorics, and play important roles in
several branches of mathematics, such as coding theory, graph theory and experimental designs (cf. \cite{Bailey,bbit,bi}).

Let $d$ be a nonnegative integer. A symmetric $d$-class \emph{association scheme} $(X, \{R_i\}_{i=0}^d)$ consists of a nonempty finite set $X$ together with $d+1$ nonempty binary relations $R_0, R_1, \ldots, R_d$ such that:
\begin{itemize}
  \item[$(1)$] $R_0=\{(x,x):x\in X\}$ is the identity relation, and the relations $R_0, \dots, R_d$ form a partition of $X \times X$;
  \item[$(2)$] each relation is symmetric, i.e., $(x, y) \in R_i$ implies $(y, x) \in R_i$;
  \item[$(3)$] for any $i,j,k\in\{0,1,\ldots,d\}$ and any pair $(x, y) \in R_k$, the number of elements $z\in X$ such that $(x, z) \in R_i$ and $(z, y) \in R_j$ is a constant $p_{ij}^k$, which depends only on $i, j,$ and $k$, and is independent of the choice of $(x, y)$.
\end{itemize}
The numbers $p_{ij}^k$ are called the \emph{intersection numbers} of the scheme. Two elements $x,y\in X$ are called \emph{$i$-th associates} with $i \in \{0,1, \dots, d\}$ if $(x,y) \in R_i$.

Let $A_i$ denote the adjacency matrix of the relation $R_i$ with $A_i(x,y)=1$ for $(x,y) \in R_i$ and $0$ otherwise, where $A_i(x,y)$ represents the entry in the $x$-th row and $y$-th column of $A_i$. Then $A_0=I$ and $\sum_{i=0}^{d} A_i=J$, where $I$ and $J$ denote the identity matrix and the all-ones matrix, respectively. Let
$\mathcal{A}=\{\sum_{i=0}^{d} a_i A_i : a_0, \dots, a_d \in \mathbb{R}\}.$
By the definition of the intersection numbers $p_{ij}^k$, we have $A_i A_j = \sum_{k=0}^{d} p_{ij}^k A_k$ for all $i,j\in\{0,1,\ldots,d\}$, and the symmetry of each matrix $A_i$ implies $p_{ij}^k=p_{ji}^k$ for all $i,j,k\in\{0,1,\ldots,d\}$. Thus $\mathcal{A}$ is a commutative algebra of dimension $d+1$ generated by $A_0,A_1,\ldots,A_d$, which is called the \emph{Bose-Mesner algebra} of the association scheme.

Since $A_0,A_1,\ldots,A_d$ are pairwise commuting symmetric matrices, they can be simultaneously diagonalized by an orthogonal matrix. Consequently, $\mathbb{R}^X$, the real vector
space whose coordinates are indexed by $X$, decomposes into a direct sum of maximal common eigenspaces of $A_0,A_1,\ldots,A_d$ (cf. \cite[Section 2.3]{bi}):
\begin{align}\label{eq:R_X}
\mathbb{R}^X=U_0\oplus U_1\oplus\cdots\oplus U_d,
\end{align}
where $U_0$ is the $1$-dimensional subspace spanned by the all-ones vector $\mathbf{1}=(1,1,\ldots,1)^\top$. Note that the definition of $p_{ii}^0$ implies
\begin{align}\label{eq:p_ii^0}
A_i \mathbf{1} = p_{ii}^0 \mathbf{1}
\end{align}
for each $i\in\{0,1,\ldots,d\}$.

For $i\in \{0,1,\ldots,d\}$, let $E_i$ denote the orthogonal projector from $\mathbb{R}^X$ onto $U_i$, represented as a symmetric matrix with respect to the standard unit orthonormal basis of $\mathbb{R}^X$ (i.e., the basis consisting of vectors with exactly one entry $1$ and all other entries $0$). Then $E_iE_j = \delta_{ij}E_i$ for all $i,j \in \{0,1,\ldots,d\}$, where $\delta_{ij}=1$ if $i=j$ and $0$ otherwise. Since $E_i^2=E_i$, we have
\begin{align}\label{eq:U_i}
\dim U_i = \operatorname{rank}(E_i)=\operatorname{tr}(E_i) \quad \text{(cf. \cite[Lemma 2.1]{Bailey})}.
\end{align}
Especially, $E_0=J/|X|$ (cf. \cite[Lemma 2.8]{Bailey}).
Furthermore, $E_0,E_1,\ldots,E_d$ are the primitive idempotents of the Bose-Mesner algebra $\mathcal{A}$ and form a basis for $\mathcal{A}$ (cf. \cite[Theorem 2.6]{Bailey} or \cite[Theorem 3.1]{bi}).

For any $M\in \mathcal{A}$, let
\begin{align}\label{eq:ME_i}
 M=\sum_{i=0}^d \mu_i E_i \text{ (note that these }\mu_i\text{ are not necessarily distinct).}
\end{align}
Then $ME_j=\mu_j E_j$ for any $j\in \{0,1,\ldots,d\}$, and thus all columns of $E_j$ are eigenvectors of $M$ associated with the eigenvalue $\mu_j$. Since $E_j^2=E_j$, all columns of $E_j$ lie in $U_j$. By \eqref{eq:U_i}, $\dim U_j = \operatorname{rank}(E_j)$, it follows that for each $j\in \{0,1,\ldots,d\}$,
\begin{align}\label{eq:M}
U_j \text{ is contained in an eigenspace of } M \text{ corresponding to the eigenvalue } \mu_j.
\end{align}
Since $\sum_{i=0}^d E_i=I$ (cf. \cite[Lemma 2.1]{Bailey}), if all $\mu_i$ are nonzero, then $M$ is invertible and
\begin{align}\label{eq:M^-1}
M^{-1} = \sum_{i=0}^d \frac{1}{\mu_i} E_i.
\end{align}

Given that $\{A_0,A_1,\ldots,A_d\}$ and $\{E_0,E_1,\ldots,E_d\}$ are both bases of the Bose-Mesner algebra $\mathcal{A}$, there exist unique $(d+1)\times (d+1)$ matrices $C=(C(i,j))$ and $D=(D(i,j))$ for $i,j\in\{0,1,\ldots,d\}$ with $C=D^{-1}$, such that
\begin{align*}
A_i=\sum_{j=0}^{d} C(i,j) E_j \quad \text{and} \quad E_i=\sum_{j=0}^{d} D(i,j) A_j
\end{align*}
for each $i\in\{0,1,\ldots,d\}$. We refer to $C$ and $D$ as the \emph{first} and \emph{second eigenmatrices} of the scheme, respectively.

Some techniques for computing $C$ and $D$ are introduced in \cite[Section~2.4]{Bailey}. Here we present one such technique, which will be used in Section \ref{subsec:asr}. Let $M=\sum_{k=0}^{d} a_k A_k\in\mathcal{A}$. Suppose that $M$ has exactly $d+1$ distinct eigenvalues $\lambda_0,\lambda_1,\ldots,\lambda_{d}$. Then $M$ has exactly $d+1$ distinct eigenspaces. By \eqref{eq:M}, every eigenspace of $M$ is a direct sum of certain $U_j$ for $j\in\{0,1,\ldots,d\}$. Thus the $d+1$ distinct eigenspaces of $M$ are precisely $U_0,U_1,\ldots,U_{d}$, and we can therefore apply \cite[Formula (2.1)]{Bailey} to obtain
\begin{align}\label{eq:E_i}
E_i = \displaystyle\prod_{\substack{\ell=0\\ \ell\neq i}}^{d} \frac{(M-\lambda_\ell I)}{(\lambda_i-\lambda_\ell)}
\end{align}
for each $i\in\{0,1,\ldots,d\}$. Substituting $M=\sum_{k=0}^{d} a_k A_k$ into \eqref{eq:E_i}, we may express each $E_i$ as a linear combination of $A_0,A_1,\ldots,A_d$. This gives the matrix $D$, and its inverse then yields the matrix $C$.

\subsection{Association schemes from balanced complete $r$-partite graphs}\label{subsec:asr}

This section focuses on the association scheme defined on the edge set of a balanced complete $r$-partite graph, where $r\geq4$.

Let $r \geq 4$ be an integer and let $\Gamma$ be a balanced complete $r$-partite graph on $rn$ vertices. Define the binary relations $R_0, \dots, R_5$ on the edge set $E(\Gamma)$ as follows:
\begin{itemize}
    \item[$R_0$:] two edges are identical (the identity relation);
    \item[$R_1$:] two edges share the same two parts and exactly one vertex;
    \item[$R_2$:] two edges share the same two parts but share no vertices;
    \item[$R_3$:] two edges share exactly one part and one vertex;
    \item[$R_4$:] two edges share exactly one part but no vertices;
    \item[$R_5$:] two edges share no parts.
\end{itemize}
Then $\left( E(\Gamma), \{R_i\}_{i=0}^5 \right)$ is a symmetric $5$-class association scheme, where the explicit values of the intersection numbers $p_{ij}^k$ are listed in Appendix \ref{app:A}. Let $A_i$ be the adjacency matrix of $R_i$, and $\mathcal A = \{\sum_{i=0}^{5} a_i A_i : a_0, \dots, a_5 \in \mathbb{R} \}$
be the Bose-Mesner algebra of the scheme.

Recall from \eqref{eq:M_H} that $M_\Gamma$ is a matrix whose rows and columns are indexed by $E(\Gamma)$, and the $(e,e')$-entry counts the number of copies of $K_s$ in $\Gamma$ that contain both edges $e$ and $e'$. In $\Gamma$, any two edges that are $0$-th associates appear together in exactly $\binom{r-2}{s-2} n^{s-2}$ copies of $K_s$, those that are $3$-rd associates appear in exactly $\binom{r-3}{s-3} n^{s-3}$ copies of $K_s$, and those that are $5$-th associates appear in exactly $\binom{r-4}{s-4} n^{s-4}$ copies of $K_s$; otherwise, no $K_s$ contains both edges. Therefore,
\begin{align}\label{eq:M_gamma}
M_\Gamma = \binom{r-2}{s-2}n^{s-2}A_0 + \binom{r-3}{s-3}n^{s-3}A_3 + \binom{r-4}{s-4}n^{s-4}A_5 \in \mathcal{A}.
\end{align}

Let $E_0,E_1,\ldots,E_5$ be the primitive idempotents of the Bose-Mesner algebra $\mathcal{A}$, where $E_0 = J/|E(\Gamma)|=J/(\binom{r}{2}n^2)$. To determine the eigenvalues of $M_{\Gamma}$, we shall express $M_{\Gamma}$ as a linear combination of $E_i$ for $i\in\{0,1,\ldots,5\}$. Then by \eqref{eq:ME_i} and \eqref{eq:M}, the coefficients of $E_i$ in this linear combination are precisely its eigenvalues. Therefore, by \eqref{eq:M_gamma}, we need to express each of $A_0$, $A_3$, and $A_5$ as a linear combination of $E_i$ for $i\in\{0,1,\ldots,5\}$. To this end, we shall compute the first eigenmatrix of the scheme. Before that, we need the following lemma.

\begin{Lemma}\label{lem:A2+A3}
Let $r\geq 4$. The matrix $2r^2A_0 + r^2 A_1+A_3$ has exactly six distinct eigenvalues:
\begin{table}[h]\centering\renewcommand{\arraystretch}{1.5}
\begin{tabular}{c|c|c|c|c|c}\toprule
$\lambda_0$ & $\lambda_1$ & $\lambda_2$ & $\lambda_3$ & $\lambda_4$ & $\lambda_5$ \\ \hline
$2n(r^2+r-2)$ & $n(2r^2+r-4)$ & $2n(r^2-1)$ & $n(r^2+r-2)$ & $n(r^2-1)$ & $0$ \\
\bottomrule
\end{tabular} .
\end{table}
\end{Lemma}

\begin{proof}
For $0 \leq k \leq 5$, define the $6 \times 6$ matrix $B_k$ by $B_k(i,j) = p_{jk}^i = p_{kj}^i$, where $0 \leq i, j \leq 5$. Let $\mathcal{B} = \{ \sum_{k=0}^{5} a_k B_k : a_0,\dots,a_5 \in \mathbb{R}\}$. Define the mapping $f : \mathcal{A} \to \mathcal{B}$ by $f (\sum_{k=0}^{5} a_k A_k) = \sum_{k=0}^{5} a_k B_k$. By \cite[Theorem 2.16]{Bailey}, $\mathcal{B}$ forms an algebra and $f$ is an algebra isomorphism from $\mathcal{A}$ to $\mathcal{B}$. In particular, $f$ maps the identity matrix $A_0 \in \mathcal{A}$ to the identity matrix $B_0 \in \mathcal{B}$. Consequently, every matrix $M \in \mathcal{A}$ and its image $f(M) \in \mathcal{B}$ have the same minimal polynomial, and hence they share the same set of distinct eigenvalues.

Since $2r^2B_0 + r^2 B_1+B_3=$
$$\setlength{\arraycolsep}{2.6pt}
\begin{pmatrix}
2r^2 & 2r^2(n - 1) & 0 & 2(r - 2)n & 0 & 0 \\
r^2 & r^2n  & r^2(n - 1) & (r - 2)n & (r - 2)n & 0 \\
0 & 2r^2 & 2r^2(n - 1) & 0 & 2(r - 2)n & 0 \\
1 & n - 1 & 0 & (r^2+r-3)n+1+r^2 & (r^2+1)(n-1) & (r - 3)n \\
0 & 1 & n - 1 & r^2 + 1 & (2r^2+r-2)n-1-r^2 & (r - 3)n \\
0 & 0 & 0 & 4 & 4(n - 1)  & 2(r^2+r-4)n
\end{pmatrix},
$$
it is readily checked that $2r^2B_0 + r^2 B_1+B_3$ has six distinct eigenvalues as given in the statement of this lemma, and thus so does $2r^2A_0 + r^2 A_1+A_3$.
\end{proof}

\begin{Lemma}\label{lem:CD}
Let $C=(C(i,j))$ and $D=(D(i,j))$ be the first and second eigenmatrices of the scheme $\left( E(\Gamma), \{R_\ell\}_{\ell=0}^5 \right)$, respectively, where $i,j\in\{0,1,\ldots,5\}$. Then for each $i \in \{0,1, \dots, 5\}$,
\begin{align*}
A_i=\sum_{j=0}^{5} C(i,j) E_j \quad \text{and} \quad E_i=\sum_{j=0}^{5} D(i,j) A_j,
\end{align*}
where
$$ C=
\begin{pmatrix}
1 & 1 & 1 & 1 & 1 & 1 \\
2(n-1) & 2(n-1) & 2(n-1) & n-2 & n-2 & -2 \\
(n-1)^2 & (n-1)^2 & (n-1)^2 & 1-n & 1-n & 1 \\
2(r-2)n & (r-4)n & -2n & (r-2)n & -n & 0 \\
2(r-2)(n-1)n & (r-4)(n-1)n & 2(1-n)n & (2-r)n & n & 0 \\
\binom{r-2}{2}n^2 & (3-r)n^2 & n^2 & 0 & 0 & 0 \\
\end{pmatrix}
$$
and
{\renewcommand{\arraystretch}{1.2}
$$
D= \frac{1}{\binom{r}{2}n^2}
\begin{pmatrix}
1 & 1 & 1 & 1 & 1 & 1 \\
r-1 & r-1 & r-1 & \frac{(r-4)(r-1)}{2(r-2)} & \frac{(r-4)(r-1)}{2(r-2)} & \frac{2(1-r)}{r-2} \\
\frac{r(r-3)}{2} & \frac{r(r-3)}{2} & \frac{r(r-3)}{2} & \frac{r(3-r)}{2(r-2)} & \frac{r(3-r)}{2(r-2)} & \frac{r}{r-2} \\
r(n-1) & \frac{r(n-2)}{2} & -r & \frac{r(n-1)}{2} & -\frac{r}{2} & 0 \\
r(r-2)(n-1) & \frac{r(r-2)(n-2)}{2} & r(2-r) & \frac{r(1-n)}{2} & \frac{r}{2} & 0 \\
\binom{r}{2}(n-1)^2 & \binom{r}{2}(1-n) & \binom{r}{2} & 0 & 0 & 0 \\
\end{pmatrix}.
$$}
\end{Lemma}

\begin{proof}
By Lemma \ref{lem:A2+A3}, the matrix $2r^2A_0 + r^2 A_1+A_3$ has exactly six distinct eigenvalues $\lambda_0,\lambda_1,\ldots,\lambda_{5}$. It follows from \eqref{eq:E_i} that
$$
E_i = \displaystyle\prod_{\substack{\ell=0\\ \ell\neq i}}^{5} \frac{(2r^2A_0 + r^2 A_1+A_3-\lambda_\ell I)}{(\lambda_i-\lambda_\ell)},\quad i=0,\dots,5.
$$
Expanding $E_i$ via the equations $A_i A_j = \sum_{k=0}^{5} p_{ij}^k A_k$, we may express each $E_i$ as a linear combination of $A_0,A_1,\ldots,A_5$. This gives the matrix $D$, and its inverse then yields the matrix $C$. We here omit the tedious computation details. Readers may verify this for themselves.
\end{proof}

\begin{Proposition}\label{prop:spectrum}
Let $3\leq s<r$. The eigenvalues of $M_\Gamma$, their multiplicities, and the corresponding eigenspaces are shown in the following table.
\begin{table}[H]
\centering
\renewcommand{\arraystretch}{1.4}
\begin{tabular}{c|c|c}\toprule
Eigenvalue & Multiplicity & Eigenspace \\\hline
 $\lambda_0=\frac{s(s-1)}{2} \binom{r-2}{s-2} n^{s-2}$    & $1$ & $U_0=\langle \mathbf{1} \rangle$ \\
 $\lambda_1=\frac{(r-s)(s-1)}{r-2} \binom{r-2}{s-2}   n^{s-2}$    &  $r-1$ & $U_1$ \\
 $\lambda_2=\frac{(r-s-1)(r-s)}{(r-2)(r-3)} \binom{r-2}{s-2} n^{s-2}$    &  $\frac{r(r-3)}{2}$ & $U_2$  \\
 $\lambda_3=(s-1) \binom{r-2}{s-2}  n^{s-2}$    &  $r(n-1)$ & $U_3$ \\
 $\lambda_4=\frac{r-s}{r-2} \binom{r-2}{s-2} n^{s-2}$  &  $r(r-2)(n-1)$ & $U_4$  \\
 $\lambda_5=\binom{r-2}{s-2}n^{s-2}$    &  $\binom{r}{2}(n-1)^2$  & $U_5$ \\
\bottomrule
\end{tabular}
\end{table}
\end{Proposition}

\begin{proof}
By \eqref{eq:M_gamma}, $M_{\Gamma} = \binom{r-2}{s-2}n^{s-2}A_0 + \binom{r-3}{s-3}n^{s-3}A_3 + \binom{r-4}{s-4}n^{s-4}A_5$. Applying Lemma \ref{lem:CD}, we may express $A_0,A_3,A_5$ as linear combinations of $E_0,E_1,\ldots,E_5$. Substituting into $M_\Gamma$, we obtain
\begin{align}\label{eq:M_Gamma E}
M_\Gamma = \binom{r-2}{s-2} n^{s-2} \Big( & \frac{s(s-1)}{2} E_0 + \frac{(r-s)(s-1)}{r-2} E_1 + \frac{(r-s-1)(r-s)}{(r-2)(r-3)} E_2 \notag \\
& + (s-1) E_3 + \frac{r-s}{r-2} E_4 + E_5 \Big).
\end{align}
By \eqref{eq:ME_i} and \eqref{eq:M}, the coefficients of $E_i$ are the eigenvalues of $M_\Gamma$. Furthermore, by \eqref{eq:U_i},
$$\dim U_i = \operatorname{tr}(E_i) = \operatorname{tr}(\sum_{j=0}^{5} D(i,j) A_j)= D(i,0) \operatorname{tr}(A_0)=D(i,0) |E(\Gamma)|= D(i,0) \binom{r}{2}n^2.$$
Therefore, multiplying the first column of $D$ in Lemma \ref{lem:CD} by $\binom{r}{2}n^2$ gives the multiplicity of each eigenvalue.
\end{proof}

\section{Proof of Theorem~\ref{thm:main1}}\label{sec:s+2}

This section establishes the minimum degree threshold for a balanced $r$-partite graph $G$ to admit a fractional $K_s$-decomposition when $r\geq s+2$. By Lemma \ref{lem:frac1}, it suffices to show that the linear system $(M_\Gamma + \Delta M_{(\Gamma,G)})\mathbf{z} = \mathbf{1}$ admits a nonnegative solution $\mathbf{z} \geq 0$. This can be established by applying Lemma \ref{prop:fromdukes}. Prior to that, we need to compute the infinity norms of the matrices $M_{\Gamma}^{-1}$ and $\Delta M_{(\Gamma,G)}$.

\begin{Lemma}\label{prop:norm1}
Let $s\geq3$ and $r \geq s+2$. Then
$$
\| M_\Gamma^{-1} \|_\infty = \frac{2 \left( r^2(2s^2-4s+1)-r(12s^2-26s+9)+(17s^2-39s+16) \right) }{s(s-1)(r-2)(r-s-1)\binom{r-3}{s-2}n^{s-2}}.
$$
\end{Lemma}

\begin{proof}
By Proposition \ref{prop:spectrum}, $M_{\Gamma}$ has eigenvalues $\lambda_0,\lambda_1,\ldots,\lambda_5$, which are all positive for $s\geq 3$ and $r \geq s+2$. It follows from \eqref{eq:M^-1} and Lemma~\ref{lem:CD} that
$$ M_\Gamma^{-1} =  \sum_{i=0}^5 \frac{1}{\lambda_i} E_i
= \sum_{i=0}^5 \frac{1}{\lambda_i} \left( \sum_{j=0}^5  D(i,j)  A_j \right)
= \sum_{j=0}^5 \left( \sum_{i=0}^5 \frac{1}{\lambda_i} D(i,j) \right)  A_j.$$
By \eqref{eq:p_ii^0}, for each $j\in\{0,1,\ldots,5\}$, all row sums of $A_j$ are equal to $p_{jj}^0$. Since each $A_j$ is a 0-1 matrix and no two matrices $A_j$ and $A_{k}$ ($j\neq k$) share a common position with value $1$, we have
$$
\begin{aligned}
\|M_\Gamma^{-1} \|_\infty &= \sum_{j=0}^5 \left| \sum_{i=0}^5 \dfrac{1}{\lambda_i} D(i,j)\right| p_{jj}^0\\
&= \frac{2 \left( r^2(2s^2-4s+1)-r(12s^2-26s+9)+(17s^2-39s+16) \right) }{s(s-1)(r-2)(r-s-1)\binom{r-3}{s-2}n^{s-2}},
\end{aligned}
$$
where the values $p_{jj}^0$ come from Appendix \ref{app:A}.
\end{proof}

\begin{Lemma}\label{prop:norm2}
Let $s\geq3$ and $r \geq s+1$. If $\hat\delta(G) \geq (1-c)n$ for some $0\leq c\leq 1$, then
$$\| \Delta M_{(\Gamma,G)} \|_\infty \leq  \dfrac{cs(s-1)(s+1)(r-2)}{4} \binom{r-3}{s-3}  n^{s-2}.$$
\end{Lemma}

\begin{proof}
Recall from \eqref{eq:Delta_M} that
$$-\Delta M_{(\Gamma,G)} =
\left(
\begin{array}{c:c}
 M_\Gamma[E(G),E(G)]-M_G &  M_\Gamma[E(G),\overline{E(G)}] \\
\hdashline
\\[-2ex]\multicolumn{2}{c}{ O }
\end{array}\right).
$$
Then $\| \Delta M_{(\Gamma,G)} \|_\infty = \|-\Delta M_{(\Gamma,G)} \|_\infty$.
For $e\in E(G)$ and $e' \in E(\Gamma)$, the $(e,e')$-entry of the matrix $-\Delta M_{(\Gamma,G)}$ counts the number of copies of $K_s$ in $\Gamma$ that contain both edges $e$ and $e'$, and that are not present in $G$. Thus, given $e\in E(G)$, the row sum of $-\Delta M_{(\Gamma,G)}$ indexed by $e$ is
\begin{align}\label{eq:Delta,G}
&\sum_{e'\in E(\Gamma)}\sum_{\substack{K\in\mathcal{K}_s(\Gamma)\\ \text{$K$ is not a subgraph of $G$} }} 1_{\{e,e'\}\subseteq E(K)} \notag \\
=&\sum_{\substack{K\in\mathcal{K}_s(\Gamma)\\
\text{$K$ is not a subgraph of $G$} }} \sum_{e'\in E(\Gamma)} 1_{\{e,e'\}\subseteq E(K)}
= \sum_{\substack{K\in\mathcal{K}_s(\Gamma),\ e\in E(K)\\ \text{$K$ is not a subgraph of $G$} }} \binom{s}{2},
\end{align}
where $1_{\{e,e'\}\subseteq E(K)}=1$ if $e$ and $e'$ are both edges of $K$ (allowing $e=e'$), and $0$ otherwise. Under the summation sign, $K$ is not a subgraph of $G$, and hence such a $K$ contains an edge not in $G$, denoted $e''$. If $e$ and $e''$ are $3$-rd associates, then there are $\binom{r-3}{s-3}n^{s-3}$ copies of $K_s$ in $\Gamma$ that contain both $e$ and $e''$; if $e$ and $e''$ are $5$-th associates, then there are $\binom{r-4}{s-4}n^{s-4}$ copies of $K_s$ in $\Gamma$ that contain both $e$ and $e''$; otherwise, there is no copy of $K_s$ in $\Gamma$ that contains both $e$ and $e''$. Since $\hat\delta(G)\geq (1-c)n$, the number of edges in $\Gamma$ that are not in $G$ and are $3$-rd associates of $e$ is at most $2(r-2)cn$; the number of edges in $\Gamma$ that are not in $G$ and are $5$-th associates of $e$ is at most $\binom{r-2}{2}cn^2$. Therefore
$$
\begin{aligned}
\sum_{\substack{K\in\mathcal{K}_s(\Gamma),\ e\in E(K)\\ \text{$K$ is not a subgraph of $G$} }} \binom{s}{2}
& \leq \binom{s}{2} \left( 2(r-2)cn\binom{r-3}{s-3}n^{s-3} + \binom{r-2}{2}cn^2\binom{r-4}{s-4}n^{s-4} \right) \\
& = \frac{cs(s-1)(s+1)(r-2)}{4} \binom{r-3}{s-3}  n^{s-2},
\end{aligned}
$$
which yields the desired result by \eqref{eq:Delta,G}.
\end{proof}

\begin{Lemma}\label{prop:solution1}
Let $s\geq3$ and $r \geq s+2$. If $\hat\delta(G) \geq (1-c)n$, where
\begin{align}\label{eq:c-1}
c\leq \frac{(r-s)(r-s-1)}{(s-2)(s+1) \left( r^2(2s^2-4s+1) + r(-12s^2+26s-9) + (17s^2-39s+16) \right)},
\end{align}
then the linear system $(M_\Gamma + \Delta M_{(\Gamma,G)} )\mathbf{z} = \mathbf{1}$ admits a solution $\mathbf{z}\geq 0$.
\end{Lemma}

\begin{proof}
It suffices to examine the three conditions in Lemma \ref{prop:fromdukes}. By Proposition~\ref{prop:spectrum}, when $s\geq 3$ and $r \geq s+2$, the eigenvalues of $M_{\Gamma}$ are all positive, and hence $M_\Gamma$ is invertible. Proposition \ref{prop:spectrum} also implies that the linear system $M_\Gamma\mathbf{x} = \mathbf{1}$ admits a solution $\mathbf{x}=\frac{1}{\lambda_0}\mathbf{1}$, where $\frac{1}{\lambda_0}>0$. By Lemmas \ref{prop:norm1} and \ref{prop:norm2},
\begin{align*}
&\| M_\Gamma^{-1} \Delta M_{(\Gamma,G)} \|_\infty
\leq \| M_\Gamma^{-1} \|_\infty \| \Delta M_{(\Gamma,G)} \|_\infty \\
&
\ \ \ \ \ \ \ \leq \frac{c(s - 2)(s + 1) \left( r^2(2s^2-4s+1) + r(-12s^2+26s-9) + (17s^2-39s+16) \right) } {2(r - s)(r-s-1)}\leq \frac{1}{2}.
\end{align*}
Thus we can apply Lemma \ref{prop:fromdukes} to complete the proof.
\end{proof}

\begin{proof}[\textbf{\textup{Proof of Theorem~\ref{thm:main1}}}]
It is readily checked that for any $s\geq 3$ and $r\geq s+2$,
$$\frac{1}{(s-2)(s+1)(s-1)^4} \leq \text{the right side of }\eqref{eq:c-1}.$$
The proof is completed by combining Lemmas \ref{lem:frac1} and \ref{prop:solution1}.
\end{proof}

\section{The case $r=s+1$}\label{sec:s+1}

Throughout this section, let $s\geq3$ be an integer and let $\Gamma$ be a balanced complete $(s+1)$-partite graph on $(s+1)n$ vertices with vertex partition $\mathcal{P} = \{V_1,V_2, \dots, V_{s+1}\}$. Let $G$ be an $s$-admissible balanced $(s+1)$-partite subgraph of $\Gamma$ with vertex partition $\mathcal{P}$. This section establishes the minimum degree threshold for $G$ to admit a fractional $K_s$-decomposition.

For $r=s+1$, Proposition \ref{prop:spectrum} shows that $M_\Gamma$ has a zero eigenvalue $\lambda_2 = 0$, so $M_\Gamma$ is non-invertible. Thus we cannot use Lemma \ref{prop:fromdukes} to obtain a solution $\mathbf{z} \ge 0$ to the linear system $(M_\Gamma + \Delta M_{(\Gamma,G)})\mathbf{z} = \mathbf{1}$, and consequently cannot follow the same approach as in Section~\ref{sec:s+2} (which handled $r \ge s+2$) to prove that $G$ admits a fractional $K_s$-decomposition.

We use a slightly different strategy. By Lemma \ref{lem:frac}, $G$ admits a fractional $K_s$-decomposition if and only if the linear system $M_G \mathbf{y} = \mathbf{1}$ has a solution $\mathbf{y} \geq \mathbf{0}$. The following lemma transforms the problem of solving the linear system $M_G \mathbf{y} = \mathbf{1}$ into the problem of solving the linear system $(M_G + \eta E_2[E(G), E(G)])\mathbf{y} = \mathbf{1}$, where $\eta$ is a nonzero real number and $E_2$ is the orthogonal projector onto $U_2$.

\begin{Lemma}\label{cor:fromdukes4}
Let $\eta$ be a nonzero real number. If $M_G+ \eta E_2[E(G),E(G)]$ is invertible, then the unique solution $\mathbf{y}$ to the linear system $(M_G+ \eta E_2[E(G),E(G)])\mathbf{y}=\mathbf{1}$ is also a solution to the linear system $M_G\mathbf{y} = \mathbf{1}$.
\end{Lemma}

\subsection{Proof of Lemma \ref{cor:fromdukes4}}\label{sec:kernel}

Let $H$ be a spanning subgraph of $\Gamma$. For $1 \leq i < j \leq s+1$, denote by $E_H(V_i, V_j)$ the set of edges in $H$ with one vertex in $V_i$ and the other vertex in $V_j$. Denote by $\mathbb{R}^{E(\Gamma)}$ the real vector space whose coordinates are indexed by $E(\Gamma)$. Let $\mathbf{1}_{H} \in \mathbb{R}^{E(\Gamma)}$ denote the indicator vector of $E(H)$, i.e., $\mathbf{1}_{H}(e)=1$ if $e\in E(H)$ and $0$ otherwise.

\begin{Lemma}\label{prop:fromdukes1}
Let $s\geq3$. Then $E_2[E(G),E(G)] \mathbf{1} = \mathbf{0}$, where $\mathbf{1}$ is the all-ones vector in $\mathbb{R}^{E(G)}$.
\end{Lemma}

\begin{proof}
Let $N$ be an $\binom{s+1}{2} \times \binom{s+1}{s}$ matrix whose rows and columns are indexed by the $2$-subsets $T$ and $s$-subsets $K$ of $\mathcal{P}$, respectively, where the $(T,K)$-entry is $1$ if $T \subseteq K$ and $0$ otherwise. Without loss of generality, assume that for each $1 \leq \ell \leq s+1$, the $\ell$-th column of $N$ is indexed by the $s$-subset $\mathcal{P} \setminus \{ V_\ell \}$. For each $1\leq \ell \leq s+1$, let $d_\ell = \sum_{k=1, k \neq \ell}^{s+1} |E_G(V_\ell, V_k)|$ and let $x_\ell = \frac{|E(G)|}{\binom{s}{2}} - \frac{d_\ell}{s-1}=\frac{2|E(G)|-sd_\ell}{s(s-1)}$. Since $G$ is $s$-admissible, by Lemma \ref{cor:nece}, $\mathbf{x}=(x_1,x_2,\ldots,x_{s+1})^\top$ satisfies $N\mathbf{x} = \mathbf{b}$, where $\mathbf{b}$ is the column vector indexed by the $2$-subsets of $\mathcal{P}$, with entries $\mathbf{b}(\{V_i,V_j\})=|E_G(V_i,V_j)|$ for $1\leq i<j\leq s+1$, and each component $x_\ell$ of $\mathbf{x}$ satisfies $x_\ell\geq 0$. Thus by repeating each edge of $G$ exactly $s(s-1)$ times, we obtain a multigraph whose edges can be partitioned into $(s+1)$ edge-disjoint subgraphs $G_1, \ldots, G_{s+1}$, where $G_\ell$, $1\leq \ell \leq s+1$ is a multipartite multigraph with vertex partition $\mathcal{P} \setminus \{V_\ell\}$, and the number of edges between any two parts in $G_\ell$ is exactly $2|E(G)|-sd_\ell$. Furthermore, we can partition the edges of $G_\ell$ into $2|E(G)| - s d_\ell$ edge-disjoint spanning subgraphs $G_{\ell,1}, \ldots, G_{\ell,\,2|E(G)| - s d_\ell}$, where for each $1 \leq \ell' \leq 2|E(G)| - s d_\ell$, the number of edges between any two parts in $G_{\ell,\ell'}$ is exactly $1$. Hence each $G_{\ell,\ell'}$ can be regarded as a spanning subgraph of $\Gamma$ satisfying the following conditions:
\begin{itemize}
\item[(i)] $|E_{G_{\ell,\ell'}}(V_{\ell}, V_j)| = 0$ for all $j\in\{1,2,\ldots,s+1\}\setminus\{\ell\}$;
\item[(ii)] $|E_{G_{\ell,\ell'}}(V_{j_1}, V_{j_2})| = 1$ for all $1 \leq j_1 < j_2 \leq s+1$ with $j_1, j_2 \neq \ell$.
\end{itemize}
Thus, the indicator vector $\mathbf{1}_{G}\in \mathbb{R}^{E(\Gamma)}$ of $E(G)$ can be decomposed as:
$$\mathbf{1}_{G} = \frac{1}{s(s-1)} \left( \sum_{\ell=1}^{s+1} \sum_{\ell'=1}^{2|E(G)| - s d_\ell} \mathbf{1}_{G_{\ell,\ell'}} \right),$$
where $\mathbf{1}_{G_{\ell,\ell'}}\in \mathbb{R}^{E(\Gamma)}$ is the indicator vector of $E(G_{\ell,\ell'})$.

We shall demonstrate that $E_2\mathbf{1}_{G_{\ell,\ell'}}=\mathbf{0}$ for every $1\leq \ell\leq s+1$ and $1 \leq \ell' \leq 2|E(G)| - s d_\ell$, which leads to $E_2 \mathbf{1}_G = \mathbf{0}$. Restricting this equality to $E(G)$ completes the proof. Therefore, it suffices to show that for any $e\in E(\Gamma)$, the inner product of the $e$-th row of $E_2$ with $\mathbf{1}_{G_{\ell,\ell'}}$ equals $0$.

By Lemma \ref{lem:CD},
\begin{align}\label{eq:E_2 expression}
E_2 = \frac{1}{\binom{s+1}{2}n^2} \left( \frac{(s+1)(s-2)}{2}(A_0+ A_1+ A_2) + \frac{(s+1)(2-s)}{2(s-1)} (A_3+ A_4) + \frac{s+1}{s-1}A_5 \right).
\end{align}
If $e\in E_{\Gamma}(V_{j_1},V_{j_2})$ for some $1\leq j_1 < j_2 \leq s+1$ and $j_1,j_2\neq \ell$, then the number of edges in $G_{\ell,\ell'}$ that are $0$-th, $1$-st, or $2$-nd associates of $e$ is $1$, the number of edges in $G_{\ell,\ell'}$ that are $3$-rd or $4$-th associates of $e$ is $2(s-2)$, and the number of edges in $G_{\ell,\ell'}$ that are $5$-th associates of $e$ is $\binom{s-2}{2}$. Hence by the right hand side of \eqref{eq:E_2 expression}, the inner product of the $e$-th row of $E_2$ with $\mathbf{1}_{G_{\ell,\ell'}}$ equals
$$\frac{1}{\binom{s+1}{2}n^2} \left( \frac{(s+1)(s-2)}{2} + \frac{(s+1)(2-s)}{2(s-1)} 2(s-2) + \frac{s+1}{s-1}\binom{s-2}{2} \right)=0.$$
If $e\in E_{\Gamma}(V_{\ell},V_{j})$ for some $1\leq j\leq s+1$ and $j\neq \ell$, then the number of edges in $G_{\ell,\ell'}$ that are $0$-th, $1$-st, or $2$-nd associates of $e$ is $0$, the number of edges in $G_{\ell,\ell'}$ that are $3$-rd or $4$-th associates of $e$ is $s-1$, and the number of edges in $G_{\ell,\ell'}$ that are $5$-th associates of $e$ is $\binom{s-1}{2}$. Hence by \eqref{eq:E_2 expression}, the inner product of $e$-th row of $E_2$ with $\mathbf{1}_{G_{\ell,\ell'}}$ equals
$$\frac{1}{\binom{s+1}{2}n^2} \left(  \frac{(s+1)(2-s)}{2(s-1)} (s-1) + \frac{s+1}{s-1}\binom{s-1}{2} \right)=0.$$
Therefore, $E_2\mathbf{1}_{G_{\ell,\ell'}}=\mathbf{0}$ for every $1\leq \ell\leq s+1$ and $1 \leq \ell' \leq 2|E(G)| - s d_\ell$.
\end{proof}

\begin{Lemma}\label{prop:fromdukes2}
Let $s\geq3$. Then $M_G E_2[E(G),E(G)]=O$.
\end{Lemma}

\begin{proof}
Recall from \eqref{eq:M_H} that $M_\Gamma = W_\Gamma W_\Gamma^\top$. For $r = s+1$, Proposition \ref{prop:spectrum} gives $\lambda_2 = 0$ and $\lambda_i\neq 0$ for all $i\neq 2$. Hence $U_2 = \ker M_\Gamma = \ker W_\Gamma^\top$. By \eqref{eq:U_i}, $U_2$ is spanned by the columns of $E_2$. Therefore \begin{align}\label{eq:E_2W_Gamma}
E_2W_\Gamma=(W_\Gamma^\top E_2^\top)^\top =(W_\Gamma^\top E_2)^\top=O.
\end{align}
Without loss of generality, assume that the elements of $E(\Gamma)$ and $\mathcal{K}_s(\Gamma)$ are ordered so that those in $E(G)$ and $\mathcal{K}_s(G)$, respectively, appear first. Then
$$
W_\Gamma=
\begin{pmatrix}
W_G & * \\
O & *
\end{pmatrix}.
$$
Let
$$
L_1 =
\begin{pmatrix}
I_{E(G) \times E(G)} \\
O_{(E(\Gamma)\setminus E(G)) \times E(G)}
\end{pmatrix}
\quad \text{and} \quad
L_2 =
\begin{pmatrix}
I_{\mathcal{K}_s(G) \times \mathcal{K}_s(G)} \\
O_{ \left( \mathcal{K}_s(\Gamma)\setminus \mathcal{K}_s(G) \right) \times \mathcal{K}_s(G)}
\end{pmatrix},
$$
where $L_1$ has rows indexed by $E(\Gamma)$ and columns indexed by $E(G)$, and $L_2$ has rows indexed by $\mathcal{K}_s(\Gamma)$ and columns indexed by $\mathcal{K}_s(G)$. Then
$$E_2[E(G),E(G)] M_G= L_1^\top E_2L_1 M_G= L_1^\top E_2L_1 W_GW_G^\top = L_1^\top E_2W_\Gamma L_2 W_G^\top \overset{\eqref{eq:E_2W_Gamma}}{=} O.$$
Since $E_2[E(G),E(G)]$ and $M_G$ are symmetric,
$M_G E_2[E(G),E(G)]= \left( E_2[E(G),E(G)]^\top M_G^\top \right)^\top$ $= \left( E_2[E(G),E(G)] M_G \right)^\top =O.$
\end{proof}

\begin{Lemma}\emph{\cite[Lemma 2.6]{BD}}\label{lem:fromdukes3}
Let $A$ and $B$ be symmetric matrices with $AB=O$ and $A+B$ invertible. If $B\mathbf{u}=\mathbf{0}$, then $A(A + B)^{-1}\mathbf{u}=\mathbf{u}$.
\end{Lemma}

Now we are in a position to give the proof of Lemma \ref{cor:fromdukes4}.

\begin{proof}[\bf Proof of Lemma \ref{cor:fromdukes4}]
By assumption, $M_G + \eta E_2[E(G),E(G)]$ is invertible, so the linear system $(M_G + \eta E_2[E(G),E(G)])\mathbf{y} = \mathbf{1}$ has a unique solution $\mathbf{y}^* = (M_G + \eta E_2[E(G),E(G)])^{-1}\mathbf{1}$. To complete the proof, it suffices to show that $M_G \mathbf{y}^* = M_G (M_G + \eta E_2[E(G),E(G)])^{-1}\mathbf{1} = \mathbf{1}$. This can be established by applying Lemma \ref{lem:fromdukes3} with $A = M_G$, $B = \eta E_2[E(G),E(G)]$, and $\mathbf{u} = \mathbf{1}$. The conditions required in Lemma \ref{lem:fromdukes3} are satisfied by Lemmas \ref{prop:fromdukes1} and \ref{prop:fromdukes2}.
\end{proof}

\subsection{Proof of Theorem~\ref{thm:main2}}

In order to apply Lemma \ref{cor:fromdukes4}, we need to determine an appropriate nonzero real number $\eta$ such that $M_G + \eta E_2[E(G),E(G)]$ is invertible and the linear system $(M_G + \eta E_2[E(G),E(G)])\mathbf{y} = \mathbf{1}$ admits a solution $\mathbf{y} \geq \mathbf{0}$. For convenience, set
$$M_G^{\eta} := M_G + \eta E_2[E(G),E(G)] \quad\text{and}\quad M_\Gamma^{\eta} := M_\Gamma + \eta E_2.$$
Define $\Delta M_{(\Gamma,G)}^{\eta}$ as the matrix with rows and columns indexed by $E(\Gamma)$ given by
\begin{align*}
\Delta M_{(\Gamma,G)}^{\eta} =
\left(
\begin{array}{c:c}
M_G^{\eta} - M_\Gamma^{\eta}[E(G),E(G)] & -M_\Gamma^{\eta}[E(G),\overline{E(G)}] \\
\hdashline
\\[-2ex]
\multicolumn{2}{c}{O}
\end{array}
\right).
\end{align*}
Then
\[
M_\Gamma^{\eta} + \Delta M_{(\Gamma,G)}^{\eta} =
\left(
\begin{array}{c:c}
M_G^{\eta} & O \\
\hdashline\\[-2ex]
\multicolumn{2}{c}{ M_\Gamma^{\eta}[\overline{E(G)},E(\Gamma)] }
\end{array}
\right).
\]
Consequently, a solution $\mathbf{z} \geq 0$ to the linear system $( M_\Gamma^{\eta} + \Delta M_{(\Gamma,G)}^{\eta} )\mathbf{z} = \mathbf{1}$ induces a solution $\mathbf{y} \geq 0$ to the linear system $M_G^{\eta} \mathbf{y} = \mathbf{1}$. Moreover, if $M_\Gamma^{\eta} + \Delta M_{(\Gamma,G)}^{\eta}$ is invertible, then so is $M_G^{\eta}$. Combining these observations with Lemma~\ref{cor:fromdukes4} yields the following result.

\begin{Lemma}\label{lem:frac2}
Let $s\geq3$ and $\eta$ be a nonzero real number. If $M_\Gamma^{\eta} + \Delta M_{(\Gamma,G)}^{\eta}$ is invertible and the linear system $( M_\Gamma^{\eta} + \Delta M_{(\Gamma,G)}^{\eta})\mathbf{z} = \mathbf{1}$ admits a solution $\mathbf{z} \geq 0$, then $G$ admits a fractional $K_s$-decomposition.
\end{Lemma}

To ensure that $M_\Gamma^{\eta} + \Delta M_{(\Gamma,G)}^{\eta}$ satisfies the conditions of Lemma~\ref{lem:frac2}, we will again use Lemma \ref{prop:fromdukes} as in Section \ref{sec:s+2}. Prior to that, we need to compute the infinity norms of the matrices $( M_\Gamma^{\eta} )^{-1}$ and $\Delta M_{(\Gamma,G)}^{\eta}$.

\begin{Lemma}\label{prop:norm4}
Let $s\geq3$ and $\eta^* = n^{s - 2}s/(s + 2)$. Then
$$\| ( M_\Gamma^{\eta^*} )^{-1} \|_\infty = \frac{2n^{-s}}{s(s-1)^2} \left( (3s^3 - 11s^2 + 12s - 3)n^2 - 2(s-1)(s-2)^2 n + (s-1)(s-2)^2 \right).$$
\end{Lemma}

\begin{proof}
By \eqref{eq:M_Gamma E} with $r=s+1$, we have
\begin{align}\label{eq:Mstar}
M_\Gamma^{\eta^*} &= M_\Gamma + \eta^* E_2 \notag
\\ & = n^{s-2} \left( \frac{s(s-1)^2}{2} E_0 + (s-1) E_1 + (s-1)^2 E_3 + E_4 + (s-1) E_5 \right) + \eta^* E_2.
\end{align}
Denote the coefficients of these $E_i$ by $\mu_i$, which are all positive for $s\geq 3$. Then by \eqref{eq:M^-1} and Lemma~\ref{lem:CD},
$$ (M_\Gamma^{\eta^*})^{-1} =  \sum_{i=0}^5 \frac{1}{\mu_i} E_i
= \sum_{i=0}^5 \frac{1}{\mu_i} \left( \sum_{j=0}^5  D(i,j)  A_j \right)
= \sum_{j=0}^5 \left( \sum_{i=0}^5 \frac{1}{\mu_i} D(i,j) \right)  A_j.$$
By \eqref{eq:p_ii^0}, for each $j\in\{0,1,\ldots,5\}$, all row sums of $A_j$ are equal to $p_{jj}^0$. Thus
$$
\begin{aligned}
 \| (M_\Gamma^{\eta^*})^{-1} \|_\infty & = \sum_{j=0}^5 \left| \sum_{i=0}^5 \dfrac{1}{\mu_i} D(i,j) \right| p_{jj}^0 \\
 & = \frac{2n^{-s}}{s(s-1)^2} \left( (3s^3 - 11s^2 + 12s - 3)n^2 - 2(s-1)(s-2)^2 n + (s-1)(s-2)^2 \right),
\end{aligned}
$$
where the values $p_{jj}^0$ come from Appendix \ref{app:A}.
\end{proof}

\begin{Lemma}\label{prop:norm6}
Let $s\geq3$. If $\hat\delta(G) \geq (1-c)n$ for some $0\leq c\leq 1$, then $\| E_2[E(G), \overline{E(G)}]  \|_\infty \leq 4(s-2)c/s$.
\end{Lemma}

\begin{proof}
By Lemma~\ref{lem:CD},
\begin{align}\label{eq:5.7}
\binom{s+1}{2} n^2 E_2 =  \frac{(s+1)(s-2)}{2} (A_0 + A_1 +A_2) - \frac{(s+1)(s-2)}{2(s-1)} (A_3 +A_4) + \frac{s+1}{s-1} A_5.
\end{align}
For any given edge $e\in E(G)$ and $i\in\{0,1,\ldots,5\}$, the row sum of $A_i[E(G),\overline{E(G)}]$ indexed by $e$ equals the number of edges in $\overline{E(G)}$ that are $i$-th associates of $e$. Since $\hat{\delta}(G)\geq (1-c)n$, the number of edges in $\overline{E(G)}$ that are $0$-th, $1$-st, or $2$-nd associates of $e$ is at most $cn^2$, the number of edges in $\overline{E(G)}$ that are $3$-rd or $4$-th associates of $e$ is at most $2(s-1)cn^2$, and the number of edges in $\overline{E(G)}$ that are $5$-th associates of $e$ is at most $\binom{s-1}{2}cn^2$. Therefore, by \eqref{eq:5.7}, we have
\[
\begin{aligned}
& \| E_2[E(G),\overline{E(G)}]  \|_\infty \\
\leq & \ \frac{1}{\binom{s+1}{2} n^2 } \left( \frac{(s+1)(s-2)}{2} cn^2 +  \frac{(s+1)(s-2)}{2(s-1)} 2(s-1)cn^2  + \frac{s+1}{s-1} \binom{s-1}{2}cn^2 \right)
= \frac{4(s-2)c}{s}.
\end{aligned}
\]
\end{proof}

\begin{Lemma}\label{prop:norm7}
Let $s\geq3$ and $\eta>0$. If $\hat\delta(G) \geq (1-c)n$ for some $0\leq c\leq 1$, then
$$\| \Delta M_{(\Gamma,G)}^{\eta} \|_\infty \leq \frac{cs(s-1)^2(s-2)(s+1)n^{s-2}}{4} + \dfrac{4(s-2)c\eta}{s}.$$
\end{Lemma}

\begin{proof}
Note that
\begin{align*}
-\Delta M_{(\Gamma,G)}^{\eta} =
\left(
\begin{array}{c:c}
M_\Gamma^{\eta}[E(G),E(G)] - M_G^{\eta} & M_\Gamma^{\eta}[E(G),\overline{E(G)}] \\
\hdashline
\\[-2ex]
\multicolumn{2}{c}{O}
\end{array}
\right).
\end{align*}
Since
$$
\begin{aligned}
M_\Gamma^{\eta}[E(G),E(G)]-M_G^{\eta} &
= (M_\Gamma+\eta E_2)[E(G),E(G)]-(M_G+ \eta E_2[E(G),E(G)]) \\
&= M_\Gamma[E(G),E(G)]-M_G,
\end{aligned}
$$
we have
\begin{align*}
-\Delta M_{(\Gamma,G)}^{\eta} & =
\left(
\begin{array}{c:c}
M_\Gamma[E(G),E(G)]-M_G & M_\Gamma[E(G),\overline{E(G)}] \\
\hdashline
\\[-2ex]
\multicolumn{2}{c}{O}
\end{array}
\right)
+
\left(
\begin{array}{c:c}
O & \eta E_2[E(G), \overline{E(G)}] \\
\hdashline
\\[-2ex]
\multicolumn{2}{c}{O}
\end{array}
\right)\\
& = -\Delta M_{(\Gamma,G)}+\left(
\begin{array}{c:c}
O & \eta E_2[E(G), \overline{E(G)}] \\
\hdashline
\\[-2ex]
\multicolumn{2}{c}{O}
\end{array}
\right),
\end{align*}
where $\Delta M_{(\Gamma,G)}$ is defined in \eqref{eq:Delta_M}. Then
\begin{align*}
\| \Delta M_{(\Gamma,G)}^{\eta}  \|_\infty=\| -\Delta M_{(\Gamma,G)}^{\eta}  \|_\infty
& \leq \| \Delta M_{(\Gamma,G)}  \|_\infty + \eta\| E_2[E(G), \overline{E(G)}]  \|_\infty \\
& \leq
\frac{cs(s-1)^2(s-2)(s+1)n^{s-2}}{4} + \frac{4(s-2)c\eta}{s},
\end{align*}
where the last inequality comes from Lemmas \ref{prop:norm2} and \ref{prop:norm6}.
\end{proof}

\begin{Lemma}\label{prop:solution2}
Let $s\geq3$ and $\eta^* = n^{s - 2}s/(s + 2)$. If $\hat\delta(G) \geq (1-c)n$, where
\begin{align}\label{eq:c-22}
c\leq  \frac{s(s-1)^2(s+2)}{(s-2)(s^5 + s^4 - 3s^3 - s^2 + 2s + 16) ( 3s^3 - 11s^2 + 12s - 3 )},
\end{align}
then $ M_\Gamma^{\eta^*} + \Delta M_{(\Gamma,G)}^{\eta^*} $ is invertible and $( M_\Gamma^{\eta^*} + \Delta M_{(\Gamma,G)}^{\eta^*} ) \mathbf{z} = \mathbf{1}$ admits a solution $\mathbf{z}\geq 0$.
\end{Lemma}

\begin{proof}
It suffices to examine the three conditions in Lemma \ref{prop:fromdukes}. By \eqref{eq:Mstar}, $M_\Gamma^{\eta^*}$ can be written as a linear combination of $E_0, E_1, \dots, E_5$, and by \eqref{eq:ME_i} and \eqref{eq:M}, the coefficients of $E_i$ are precisely the eigenvalues of $M_\Gamma^{\eta^*}$, which are all positive for any $s\geq  3$. Thus $M_\Gamma^{\eta^*}$ is invertible.
By Proposition \ref{prop:spectrum}, the linear system $M_\Gamma\mathbf{x} = \mathbf{1}$ admits a solution $\mathbf{x}=\frac{1}{\lambda_0}\mathbf{1}$, where $\frac{1}{\lambda_0}>0$. Since $E_2$ is the orthogonal projector onto $U_2$, and by \eqref{eq:R_X} the subspaces $U_2$ and $U_0=\langle \mathbf{1}\rangle$ are orthogonal, it follows that $E_2\mathbf{1}=\mathbf{0}$. Thus $\mathbf{x}=\frac{1}{\lambda_0}\mathbf{1}$ is also a solution to the linear system $M_\Gamma^{\eta^*}\mathbf{x} = (M_\Gamma + \eta^* E_2)\mathbf{x} =\mathbf{1}$.
By Lemma \ref{prop:norm4} and Lemma \ref{prop:norm7} with $\eta = \eta^*$, we have
\begin{align*}
 & \ \ \ \ \| ( M_\Gamma^{\eta^*} )^{-1} \Delta M_{(\Gamma,G)}^{\eta^*} \|_\infty
\leq \| ( M_\Gamma^{\eta^*} )^{-1} \|_\infty \| \Delta M_{(\Gamma,G)}^{\eta^*} \|_\infty \\
&
\leq \frac{c(s-2)(s^5 + s^4 - 3s^3 - s^2 + 2s + 16)( 3s^3 - 11s^2 + 12s - 3 - (\frac{2}{n}-\frac{1}{n^2})(s-1)(s-2)^2 )}{2s(s-1)^2(s+2)}\\
&\leq \frac{c(s-2)(s^5 + s^4 - 3s^3 - s^2 + 2s + 16)( 3s^3 - 11s^2 + 12s - 3  )}
{2s(s-1)^2(s+2)}\leq\frac{1}{2}.
\end{align*}
Thus we can apply Lemma \ref{prop:fromdukes} to complete the proof.
\end{proof}

Now we are in a position to give the proof of Theorem \ref{thm:main2}.

\begin{proof}[\bf Proof of Theorem \ref{thm:main2}]
It is readily checked that for any $s\geq 3$,
$$\frac{1}{3s^3(s-2)^2} \leq \text{the right side of }\eqref{eq:c-22}.$$
Apply Lemmas \ref{lem:frac2} and \ref{prop:solution2} to complete the proof.
\end{proof}

\section{Concluding remarks}\label{sec:conc}

For $s$-admissible balanced $s$-partite graphs with a large partite minimum degree and sufficiently large order, Barber, K\"{u}hn, Lo, Osthus, and Taylor \cite[Corollary 1.6]{BKLOT} proved that a fractional $K_s$-decomposition can be converted into an exact $K_s$-decomposition. In contrast, for $r>s$, no analogous result is known for $s$-admissible balanced $r$-partite graphs regarding whether a fractional $K_s$-decomposition can be converted into an exact one. This is a worthwhile direction for further research.

The existence of $s-2$ mutually orthogonal Latin squares of order $n$ is equivalent to the existence of a $K_s$-decomposition of a balanced complete $s$-partite graph on $sn$ vertices (cf. \cite[Section 1.2]{BKLOT}). A generalization of this equivalence involves the concept of group divisible designs. For a positive integer $n$ and integers $r\geq s\geq 2$, a \emph{group divisible design} is a triple $(X, \mathcal{G}, \mathcal{B})$, where $X$ is a set of $rn$ points, $\mathcal{G}$ is a partition of $X$ into $r$ \emph{groups} of size $n$, and $\mathcal{B}$ is a set of $s$-subsets of $X$ (called \emph{blocks}) such that $|B \cap G| \leq 1$ for any $B \in \mathcal{B}$ and $G \in \mathcal{G}$, and every pair of points from distinct groups is contained in exactly one block of $\mathcal{B}$. Such a design is often written as an $s$-GDD of type $r^n$, whose existence is equivalent to the existence of a $K_s$-decomposition of a balanced complete $r$-partite graph on $rn$ vertices (cf. \cite[Remark 24.7]{bs}). This paper can be viewed as a fractional version of completing a partial group divisible design (i.e., with some blocks missing) to a group divisible design, generalizing the problem of completing a partial Latin square to a Latin square.

\appendix

\section{Explicit values of the intersection numbers}\label{app:A}

This appendix provides the intersection numbers $p_{ij}^k$, $i,j,k\in\{0,1,\ldots,5\}$, for the association scheme $\left( E(\Gamma), \{R_0,R_1,\ldots,R_5\}\right)$ defined in Section~\ref{subsec:asr}. Their computation is straightforward and follows directly from the definition, so we omit the detailed calculations herein.

\begin{gather*}
\begin{array}{|c|cccccc|}\hline
p_{ij}^{0} & 0 & 1 & 2 & 3 & 4 & 5 \\
\hline
0 & 1 & 0 & 0 & 0 & 0 & 0 \\
1 & 0 & 2(n-1) & 0 & 0 & 0 & 0 \\
2 & 0 & 0 & (n-1)^2 & 0 & 0 & 0 \\
3 & 0 & 0 & 0 & 2(r-2)n & 0 & 0 \\
4 & 0 & 0 & 0 & 0 & 2(r-2)(n-1)n & 0 \\
5 & 0 & 0 & 0 & 0 & 0 & \binom{r-2}{2}n^2 \\\hline
\end{array}
\\[2ex]
\begin{array}{|c|cccccc|}\hline
p_{ij}^{1} & 0 & 1 & 2 & 3 & 4 & 5 \\
\hline
0 & 0 & 1 & 0 & 0 & 0 & 0 \\
1 & 1 & n-2 & n-1 & 0 & 0 & 0 \\
2 & 0 & n-1 & (n-1)(n-2) & 0 & 0 & 0 \\
3 & 0 & 0 & 0 & (r-2)n & (r-2)n & 0 \\
4 & 0 & 0 & 0 & (r-2)n & (r-2)(2n-3)n & 0 \\
5 & 0 & 0 & 0 & 0 & 0 & \binom{r-2}{2}n^2 \\\hline
\end{array}
\\[2ex]
\begin{array}{|c|cccccc|}\hline
p_{ij}^{2} & 0 & 1 & 2 & 3 & 4 & 5 \\
\hline
0 & 0 & 0 & 1 & 0 & 0 & 0 \\
1 & 0 & 2 & 2(n-2) & 0 & 0 & 0 \\
2 & 1 & 2(n-2) & (n-2)^2 & 0 & 0 & 0 \\
3 & 0 & 0 & 0 & 0 & 2(r-2)n & 0 \\
4 & 0 & 0 & 0 & 2(r-2)n & 2(r-2)(n-2)n & 0 \\
5 & 0 & 0 & 0 & 0 & 0 & \binom{r-2}{2}n^2 \\\hline
\end{array}
\\[2ex]
\begin{array}{|c|cccccc|}\hline
p_{ij}^{3} & 0 & 1 & 2 & 3 & 4 & 5 \\
\hline
0 & 0 & 0 & 0 & 1 & 0 & 0 \\
1 & 0 & 0 & 0 & n-1 & n-1 & 0 \\
2 & 0 & 0 & 0 & 0 & (n-1)^2 & 0 \\
3 & 1 & n-1 & 0 & (r-3)n+1 & n-1 & (r-3)n \\
4 & 0 & n-1 & (n-1)^2 & n-1 & ((r-2)n-1)(n-1) & (r-3)(n-1)n \\
5 & 0 & 0 & 0 & (r-3)n & (r-3)(n-1)n & \binom{r-3}{2}n^2 \\\hline
\end{array}
\\[2ex]
\begin{array}{|c|cccccc|}\hline
p_{ij}^{4} & 0 & 1 & 2 & 3 & 4 & 5 \\
\hline
0 & 0 & 0 & 0 & 0 & 1 & 0 \\
1 & 0 & 0 & 0 & 1 & 2n-3 & 0 \\
2 & 0 & 0 & 0 & n-1 & (n-1)(n-2) & 0 \\
3 & 0 & 1 & n-1 & 1 & (r-2)n-1 & (r-3)n \\
4 & 1 & 2n-3 & (n-1)(n-2) & (r-2)n-1 & (r-3)(n-2)n+(n-1)^2 & (r-3)(n-1)n \\
5 & 0 & 0 & 0 & (r-3)n & (r-3)(n-1)n & \binom{r-3}{2}n^2 \\\hline
\end{array}
\\[2ex]
\begin{array}{|c|cccccc|}\hline
p_{ij}^{5} & 0 & 1 & 2 & 3 & 4 & 5 \\
\hline
0 & 0 & 0 & 0 & 0 & 0 & 1 \\
1 & 0 & 0 & 0 & 0 & 0 & 2(n-1) \\
2 & 0 & 0 & 0 & 0 & 0 & (n-1)^2 \\
3 & 0 & 0 & 0 & 4 & 4(n-1) & 2(r-4)n \\
4 & 0 & 0 & 0 & 4(n-1) & 4(n-1)^2 & 2(r-4)(n-1)n \\
5 & 1 & 2(n-1) & (n-1)^2 & 2(r-4)n & 2(r-4)(n-1)n & \binom{r-4}{2}n^2 \\\hline
\end{array}
\end{gather*}


\begin{thebibliography}{99}

\bibitem{Bailey}
R.A.~Bailey, Association Schemes: Designed Experiments, Algebra and Combinatorics, Cambridge Univ. Press, Cambridge, 2004.

\bibitem{bbit}
E.~Bannai, E.~Bannai, T.~Ito, and R.~Tanaka, Algebraic Combinatorics, De Gruyter, Berlin, Boston, 2021.

\bibitem{bi}
E.~Bannai and T.~Ito, Algebraic Combinatorics I: Association Schemes, Benjamin/Cum-mings, Menlo Park, California, 1984.

\bibitem{BKLMO}
B.~Barber, D.~K\"{u}hn, A.~Lo, R.~Montgomery, and D.~Osthus, Fractional clique decompositions of dense graphs and hypergraphs, J. Comb. Theory, Ser. B, 127 (2017), 148--186.

\bibitem{BKLO}
B.~Barber, D.~K\"{u}hn, A.~Lo, and D.~Osthus, Edge-decompositions of graphs with high minimum degree, Adv. Math., 288 (2016), 337--385.

\bibitem{BKLOT}
B.~Barber, D.~K\"{u}hn, A.~Lo, D.~Osthus and A.~Taylor, Clique decompositions of multipartite graphs and completion of Latin squares, J. Comb. Theory, Ser. A, 151 (2017), 146--201.

\bibitem{BD}
F.~Bowditch and P.~Dukes, Fractional triangle decompositions of dense $3$-partite graphs, J. Comb., 10 (2019), 255--282.

\bibitem{bs}
D.~Bryant and S.~El-Zanati, Graph decompositions, in: CRC Handbook of Combinatorial Designs (C.J.~Colbourn and J.H.~Dinitz, eds.), CRC Press, Boca Raton, 2007, 477--486.

\bibitem{Beyond}
M.~Delcourt, C.~Henderson, T.~Lesgourgues, and L.~Postle, Beyond Nash-Williams: counterexamples to clique decomposition thresholds for all cliques larger than triangles, arXiv:2508.20819v2, 2025.

\bibitem{DP21}
M.~Delcourt and L.~Postle, Progress towards Nash-Williams' conjecture on triangle decompositions, J. Comb. Theory, Ser. B, 146 (2021), 382--416.

\bibitem{Dross}
F.~Dross, Fractional triangle decompositions in graphs with large minimum degree, SIAM J. Discrete Math., 30 (2016), 36--42.

\bibitem{Dukes12}
P.~Dukes, Rational decomposition of dense hypergraphs and some related eigenvalue estimates,
Linear Algebra Appl., 436 (2012), 3736--3746; corrigendum, Linear Algebra Appl., 467 (2015), 267--269.

\bibitem{DH}
P.~Dukes and D.~Horsley, On the minimumm degree required for a triangle decomposition, SIAM J. Discrete Math., 34 (2020), 597--610.

\bibitem{Garaschuk}
K.~Garaschuk, Linear methods for rational triangle decompositions, Doctoral Dissertation, University of Victoria, 2014.


\bibitem{GKLMO}
S.~Glock, D.~K\"{u}hn, A.~Lo, R.~Montgomery, and D.~Osthus, On the decomposition threshold of a given graph, J. Comb. Theory, Ser. B, 139 (2019), 47--127.

\bibitem{GKLO}
S.~Glock, D.~K\"{u}hn, A.~Lo, and D.~Osthus, The existence of designs via iterative absorption: hypergraph $F$-designs for arbitrary $F$, Mem. Am. Math. Soc., 284 (1406), 2023.

\bibitem{HR}
P.E.~Haxell and V.~R\"{o}dl, Integer and fractional packings in dense graphs, Combinatorica, 21 (2001), 13--38.


\bibitem{Keevash14}
P.~Keevash, The existence of designs, arXiv:1401.3665, 2014.


\bibitem{Kirkman}
T.P.~Kirkman, On a problem in combinatorics, Cambridge Dublin Math. J., 2 (1847), 191--204.

\bibitem{Montgomery17}
R.~Montgomery, Fractional clique decompositions of dense partite graphs, Comb. Probab. Comput., 26 (2017), 911--943.

\bibitem{Montgomery19}
R.~Montgomery, Fractional clique decompositions of dense graphs, Random Struct. Algor., 54 (2019), 779--796.


\bibitem{Wilson75}
R.M.~Wilson, Decomposition of complete graphs into subgraphs isomorphic to a given graph, Congr. Numer., XV (1975), 647--659.

\bibitem{Yuster}
R.~Yuster, Asymptotically optimal $K_k$-packings of dense graphs via fractional $K_k$-decompositions, J. Comb. Theory, Ser. B, 95 (2005), 1--11.

\end{thebibliography}
\end{document}